\documentclass[11pt,a4paper]{article}
\usepackage{amsmath,amsthm,nicefrac}
\usepackage{amsfonts, amstext}
\usepackage{graphicx} 
\usepackage{textcomp} 
\usepackage[font={small,it}]{caption}
\usepackage{enumitem}
\usepackage{mathrsfs}
\usepackage{amsmath,amssymb, amsthm, amstext}
\usepackage{amsfonts}
\usepackage{xcolor}
\usepackage{esvect}
\usepackage[linesnumbered, ruled, vlined]{algorithm2e} 

\SetKwComment{tcp}{\textit{\% }}{} 
\SetKwInput{KwIn}{Input} 
\SetKwInput{KwOut}{Output} 
\SetKw{KwTo}{to} 
\SetKw{KwRet}{return} 
\DontPrintSemicolon 
\SetAlgoLined 
\SetNlSty{}{}{:} 
\SetAlgoNlRelativeSize{-1}
\definecolor{NearBlack}{rgb}{0.1, 0.1, 0.2}
\usepackage{hyperref}

\hypersetup{colorlinks=true,%
            citebordercolor={.1 .1 .2},linkbordercolor={.1 .1 .2},%
citecolor=NearBlack,urlcolor=black,linkcolor=NearBlack}

\newtheorem{theorem}{Theorem}[section]
\newtheorem{lemma}[theorem]{Lemma}
\newtheorem{claim}[theorem]{Claim}

\newtheorem*{lemma*}{Lemma}

\usepackage{geometry}

\title{How many colours to connect cliques?}
\author{Archit Manas}
\date{\vspace{-2ex}}

\begin{document}

\maketitle
\begin{abstract}
We study the following problem: given an edge-colouring of a $K_n$ with $r$ colours, how many colours does one need to keep (in the worst-case scenario) so that the subgraph formed by those colours is connected? We determine this extremal function up to constant factors in every parameter regime, and show that it has the same asymptotic order as the corresponding number (of colours) needed for ensuring that no vertex is isolated. This gives a complete description of the ``spanning'' threshold for connectivity in edge-coloured complete graphs.
\end{abstract}
\section{Introduction}
A folklore observation, usually attributed to Erd\H{o}s and Rado, states that in every $2$-edge-colouring of a complete graph, one of the two colour classes is connected; see, for example, \cite{GyarfasSurvey2011}. For colourings with more colours, a classical problem of Gerencs\'er and Gy\'arf\'as \cite{GerencserGyarfas1967} asks for the largest monochromatic connected component that must occur in every $r$-edge-colouring of $K_n$. \\

\noindent
Bollob\'as and Gy\'arf\'as \cite{BollobasGyarfas2008} initiated the study of large highly connected monochromatic subgraphs in edge-coloured complete graphs. Bollob\'as subsequently posed a broader ``power of many colours'' problem: given an $r$-edge-colouring of $K_n$, how large a $\kappa$-connected subgraph can one guarantee if edges of up to $s$ colours may be used? Liu, Morris and Prince obtained sharp or nearly sharp estimates in several parameter ranges, treating the multicoloured case $s\geq 2$ in \cite{LMP2008} and the monochromatic case $s=1$ in \cite{LMP2009}. \\

\noindent
The case of ordinary connectivity in the above setting ($\kappa=1$) was recently investigated systematically by Alon, Buci\'c, Christoph and Krivelevich \cite{ABCK2024}. They define $f(n,r,s)$ to be the largest integer $m$ such that every $r$-edge-colouring of $K_n$ contains a connected subgraph on at least $m$ vertices whose edges use at most $s$ colours. For $n$ sufficiently large relative to $r$, they determine the approximate behaviour of $f(n,r,s)$ throughout the parameter range, up to a logarithmic factor in the remaining regimes and more precisely in several special cases. \\

\noindent
In the same work, a weaker covering parameter $g(n,r,s)$ is also considered: the largest integer $m$ such that, in every $r$-edge-colouring of $K_n$, there is a set of at most $s$ colours incident with at least $m$ vertices. Thus, connectivity is replaced by the requirement that the graph formed by the selected colours have few isolated vertices. This problem has an equivalent formulation in terms of unions of cliques in a clique covering of $K_n$, and is related to the classical covering problems of Erd\H{o}s and R\'enyi \cite{ErdosRenyi1956}, Mills \cite{Mills1979} and F\"uredi \cite{Furedi1990}. In \cite{ABCK2024}, the parameter $g(n,r,s)$ is determined up to a constant factor for every $s$, again under the assumption that $n$ is sufficiently large relative to $r$. \\

\noindent
In this note, we study the ``spanning'' thresholds of these two parameters. Rather than fixing the number of available colours and maximising the number of vertices reached, we insist on reaching all $n$ vertices and ask for the minimum number of colours that must be allowed. \\

\noindent
More formally, let $\chi\colon E(K_n)\to[r]$ be an $r$-edge-colouring, and for $S\subseteq[r]$, let $G_{\chi}[S]$ denote the subgraph of the $K_n$ formed by the edges with colours in $S$. Define
\[
A(n,r)
:=
\max_{\chi}
\min\bigl\{
|S|:G_\chi[S]\text{ is connected}
\bigr\}
\]
and
\[
B(n,r)
:=
\max_{\chi}
\min\bigl\{
|S|:G_\chi[S]\text{ has no isolated vertices}
\bigr\}.
\]
Equivalently, in the notation of \cite{ABCK2024},
\[
A(n,r)=\min\{s\geq1:f(n,r,s)=n\},
\qquad
B(n,r)=\min\{s\geq1:g(n,r,s)=n\}.
\]
\noindent
The parameters $f$ and $g$ can have very different behaviour away from the spanning endpoint. Our main result determines $A, B$ up to constant factors and in fact shows that the two nevertheless have the same order of magnitude in every parameter regime.
\begin{theorem}
\label{thm:1:1}
Let $n,r$ be positive integers with $n\geq2$. Then
\[
A(n,r),B(n,r)
=
\Theta\left(
\min\left\{
n,\,
r,\,
\sqrt{r\log\left(2+\frac{n^2}{r}\right)}
\right\}
\right),
\]
where the implicit constants are absolute. Equivalently,
\[
A(n,r),B(n,r)
=
\begin{cases}
\Theta(n),
    & n\leq\sqrt{2r},\\[2mm]
\Theta\left(\sqrt{r\log(n^2/r)}\right),
    & \sqrt{2r}<n<
      (1.01)^r,\\[2mm]
\Theta(r),
    & n\geq
      (1.01)^r.
\end{cases}
\]
\end{theorem}
Let us remark here that throughout, we write $\log$ for the natural logarithm (to the base $e$). \\

\noindent
The parameter $A(n,r)$ is also closely related to the minimum labelling spanning tree problem, introduced by Chang and Leu \cite{ChangLeu1997}. Given an edge-labelled connected graph, the minimum labelling spanning tree problem asks for a spanning tree using the smallest possible number of distinct labels. The problem has been studied from algorithmic, approximation and optimisation perspectives; see, for example, \cite{KrumkeWirth1998}. For a fixed edge-colouring $\chi$ of $K_n$, the minimum number of colours required to form a connected spanning subgraph is exactly the minimum number of labels required by a spanning tree. Thus, the parameter $A(n,r)$ may be viewed as the largest possible optimum of the minimum labelling spanning tree problem over all $r$-edge-colourings of $K_n$. Our focus, however, is extremal rather than algorithmic: we seek bounds depending only on the number of vertices and the number of available colours.  \\

\noindent
An important distinction from the previous work on $f(n,r,s)$ and $g(n,r,s)$ is that our estimates hold uniformly for all $n$ and $r$. In particular, we determine the ``spanning'' thresholds also when the order of the complete graph is small or moderate relative to the number of colours, where the large-$n$ results of \cite{ABCK2024} do not directly apply. \\

\noindent
The proof of \hyperref[thm:1:1]{Theorem 1.1} combines two different arguments. The upper bound is obtained through a recursive reduction based on maximal monochromatic forests, while the lower bound follows from an explicit construction involving intersecting set systems. \\

\noindent
An important observation is that trees have no isolated vertices, and thus $B(n,r) \leq A(n,r)$. So it will be enough to find the appropriate upper bounds for $A$ and lower bounds for $B$.
\section{Upper Bounds for \texorpdfstring{$A(n, r)$}{A(n, r)}}
Note that picking all $r$ colours certainly connects the graph, while any spanning tree uses $n-1$ edges, and thus at most $n$ colours. And thus we have shown:
\begin{lemma}
\label{lem:2:1}
For all $n, r \geq 1$ we have $A(n, r) \leq \min(n, r)$.
\end{lemma}
\noindent The ``main'' regime of interest is $\sqrt{2r} < n < (1.01)^r$. To upper bound $A$ here, we prove the following recursive estimate.
\begin{lemma}
\label{lem:2:2}
Let $n, r, t$ be positive integers with $r \geq t$. Then we have, \[A(n, r) \leq \max \left(t-1, 1 + A \left(n - \left\lceil\frac{nt}{2r}\right\rceil, r\right) \right). \]
\end{lemma}
\begin{proof}
Consider any edge-colouring of the complete graph on the vertex set $[n]$ with colours in $[r]$ that requires one to use $A(n, r)$ colours to form a connected subgraph.
If, for some vertex, the edges incident to it span fewer than $t$ distinct colours, we can connect the graph using these colours, and conclude $A(n, r) \leq t-1$ immediately. \\\\
Thus we may now assume that each vertex is incident to at least $t$ differently coloured edges. For each colour $c \in \{1, 2, \dots, r\}$ let $\mathcal F_c$ denote a maximal forest on the subgraph of the $K_n$ consisting of edges coloured $c$. Consider the set \[S = \{(u, v, c) \in [n] \times [n] \times [r] \bigm| u \ne v, uv \in E(\mathcal F_c)\}.\]
Note that each vertex $u \in [n]$ is incident to edges spanning at least $t$ distinct colours, and for each such colour $c$ there is some vertex $v$ so that $(u, v, c) \in S$. Hence $|S| \geq nt$. \\

\noindent
On the other hand, for each colour $c$ and each edge 
$uv$ of $E(\mathcal F_c)$ there are exactly two corresponding elements in $S$, namely $(u, v, c)$ and $(v, u, c)$. Therefore, we have
\[|E(\mathcal F_1)| + \dots + |E(\mathcal F_r)| = \frac{|S|}{2} \geq \frac{nt}{2},\]
that is, there are at least $nt/2$ edges of the $K_n$ in the union of the forests $\mathcal F_c$. From the pigeonhole principle, it follows that there is some colour $c \in [r]$ for which $|E(\mathcal F_c)| \geq \lceil nt/2r \rceil$. Since $\mathcal F_c$ is a forest on $n$ vertices, it follows that it has $\leq n - \lceil \frac{nt}{2r} \rceil$ connected components. \\

\noindent
Thus, we can pick a subset $X$ of $[n]$ of size exactly $m =n - \lceil \frac{nt}{2r} \rceil$, so that each connected component has nonzero intersection with $X$. By definition, we can connect the induced subgraph on $X$ using a monochromatic subgraph on at most $A(m, r)$ colours. Adding in colour $c$ (if necessary), we deduce $A(n, r) \leq A(m, r)+1$. This establishes the lemma.
\end{proof}

\noindent
Our goal will now be to repeatedly use \hyperref[lem:2:2]{Lemma 2.2} to establish the desired upper bound on $A$.

\begin{lemma}
\label{lem:2:3}
    Let $n, r$ be positive integers with $\sqrt{2r} < n < (1.01)^r$. Then $A(n, r) \leq 5 \sqrt{r \log(n^2/r)}$.
\end{lemma}
\begin{proof}
Denote $t = \left\lceil \sqrt{r \log(n^2/r)} \right\rceil$ and note that $t \leq r$ since $n < (1.01)^r$. Define a sequence $\{n_i\}_{i \geq 0}$ given by $n_0 = n$ and $n_{i+1} = n_i - \lceil \frac{n_it}{2r} \rceil$ for all $i \geq 0$. Observe that for all $i \geq 1$, we have \[n_{i+1} \leq n_i - \frac{n_it}{2r}= n_i\left(1-\frac{t}{2r}\right) \leq n_i e^{-t/2r}\]
and thus by induction $n_k \leq n \exp(-tk/2r)$ for $k \geq 0$. Let $x$ be minimal for which $n_x \leq  \sqrt{r}$. Since we have $n_t \leq n \exp(-t^2/2r) \leq n \exp(-\log(n^2/r)/2) = \sqrt{r}$, it follows that $x \leq t$.\\

\noindent
Note that \[A(n_i, r) \leq \max(t-1, 1+A(n_{i+1}, r)) \] for all $0 \leq i < x$ from \hyperref[lem:2:2]{Lemma 2.2} applied to $n_i, r, t$.\\

\noindent
We argue by induction that for all $0 \leq i \leq x$ we have $A(n_{x-i}, r) \leq t+i+A(n_x, r)$. \\

\noindent
The base case $i =0$ is immediate. For $1 \leq i \leq x$, we have \[A(n_{x-i}, r) \le  \max(t-1, 1+A(n_{x-i+1}, r)) \leq \max(t-1, 1+t+(i-1) +A(n_x, r)) = t+i + A(n_x, r),\]
where the first estimate follows from \hyperref[lem:2:2]{Lemma 2.2} while the second estimate follows from the inductive hypotheses. We deduce that \[A(n, r) = A(n_0, r) \leq t+x+A(n_x, r) \leq t+t+n_x \leq 2 \left\lceil \sqrt{r \log(n^2/r)}\right \rceil + \sqrt{r} \leq 5\sqrt{r \log(n^2/r)}.\]
Note that we make use of $n_x \leq \sqrt{r}$, $x \leq t$ and $n > \sqrt{2r}$ in the above. This establishes the lemma.
\end{proof}
\noindent
Now observe that \hyperref[lem:2:1]{Lemma 2.1} implies the desired upper bounds for $n \leq \sqrt{2r}$ and $n \geq (1.01)^r$ since we have both $A(n, r) \leq n$ and $A(n, r) \leq r$, while \hyperref[lem:2:3]{Lemma 2.3} implies the desired upper bound $A(n, r) = O \left(\sqrt{r \log(n^2/r)}\right)$ in the regime $\sqrt{2r} < n < (1.01)^r.$ \\

\noindent
Thus we have retrieved all the desired upper bounds in \hyperref[thm:1:1]{Theorem 1.1}.
\section{Lower Bounds for \texorpdfstring{$B(n, r)$}{B(n, r)}}
We first show that $B(n, r)$ is an increasing function in both arguments.
\begin{lemma}
\label{lem:3:1} 
Let $n, r$ be positive integers. We have $B(n, r) \le B(n, r+1)$ and $B(n, r) \leq B(n+1, r)$.
\end{lemma}
\begin{proof}
Let $\chi : \binom{[n]}{2} \mapsto [r]$ be an $r$-edge-colouring of the $K_n$ for which we must use $B(n,r)$ colours to ensure that no vertex remains isolated. By interpreting $\chi$ as a function with codomain $[r+1]$, in which the final colour is unused, we get an $(r+1)$-edge colouring of $K_n$ that demonstrates $B(n,r+1) \geq B(n,r)$. \\

\noindent
For $n = 1$, it is clear that \[0 = B(1, r) \leq B(2, r) = 1\] for all $r \geq 1$. Assume $n \geq 2$ and define a colouring $\chi' : \binom{[n+1]}{2} \mapsto [r]$ by: \begin{itemize}
    \item $\chi'(\{i, j\}) = \chi(\{i, j\})$ for all $1 \leq i < j \leq n$.
    \item $\chi'(\{i, n+1\}) = \chi(\{i, n\})$ for all $1 \leq i < n$.
    \item $\chi'(\{n, n+1\}) = \chi(\{1, n\})$.
\end{itemize}
Essentially, we copy the colour pattern of edges incident on $n$ and add a new vertex with an identical colour pattern. Note that all vertices $1 \leq x \leq n$ are incident to the same set of colours in $\chi'$ as in $\chi$. Thus if fewer than $B(n,r)$ colours were to suffice in ensuring no vertex is isolated in the $K_{n+1}$ under the colouring $\chi'$, the same would be true for the $K_n$ under colouring the $\chi$. This promptly implies $B(n+1, r) \geq B(n, r)$.
\end{proof}
To show the lower bounds in \hyperref[thm:1:1]{Theorem 1.1}, we break into the three regimes for $n, r$.
\subsection*{Case I: $n \leq \sqrt{2r}$}
\noindent
In this case, we note that $\binom{n}{2} < \frac{1}{2}n^2 \leq r$, and thus we can assign distinct colours to all edges of the $K_n$. It follows that one needs at least $\lceil n/2 \rceil$  colours to ensure no vertex is isolated, since any given colour does this for at most two vertices. And thus $B(n, r) = \Omega(n)$ in this case. 
\subsection*{Case II: $\sqrt{2r} < n  < (1.01)^r$}
The lower-bound construction in this ``main'' regime is best thought of as a blow-up of the intersecting family $\binom{[2b+1]}{b+1}$, by taking $a$ copies of it. We use this family because it has many elements (roughly $(4+o(1))^b$) and also a large transversal number (equal to $b+1$). Thus $b$ supplies the exponential scale, while $a$ provides a polynomial blow-up that allows us to control the desired number of vertices and colours. In this sense, the construction interpolates between the ``square-root'' regime and the more delicate regime where $n$ is exponential in $r$. In \hyperref[sec:4]{Section 4}, one may find more motivation on where the family $\binom{[2b+1]}{b+1}$ seems to come from.\\

\noindent
Let us now introduce the formal statement of the lemma which is the key driving force behind the lower bounds in this regime.
\begin{lemma}
\label{lem:3:2}
Let $a, b$ be positive integers. Then $B\left(a\binom{2b+1}{b+1}, \binom{a+1}{2}(2b+1)\right) \geq \frac{1}{2}a(b+1)$.
\end{lemma}
\begin{proof}
Let $\mathcal F$ denote the set of all $(b+1)$-element subsets of $[2b+1]$, and let $\mathcal G$ denote the set of all $2$-element subsets of $\{0\} \cup[a]$.  Our vertex set will be given by $[a] \times \mathcal F$, while our colours will come from $\mathcal G \times [2b+1]$. \\

\noindent
The colouring $\chi : \binom{[a] \times \mathcal F}{2} \mapsto \mathcal G \times [2b+1]$ is described by:
\begin{itemize}
    \item $\chi((i, S_1), (i, S_2)) = (\{0,i\}, \min(S_1 \cap S_2))$ for all $i \in [a]$ and $S_1 \ne S_2 \in \mathcal F$.
    \item $\chi((i, S_1), (j, S_2)) = (\{i,j\}, \min(S_1 \cap S_2))$ for all $i \ne j \in [a]$ and $S_1, S_2 \in \mathcal F$.
\end{itemize}
We note that $\chi$ is well-defined since any two elements of $\mathcal F$ intersect. Let us also remark that the choice of ``$\min$'' for the second coordinate is arbitrary, the argument that follows only requires that this coordinate be contained in $S_1 \cap S_2$. \\

\noindent
Let $C \subseteq \mathcal{G} \times [2b+1]$ be a set of colours such that each vertex has an incident edge with colour in $C$, that is, no vertex is left isolated. Denote by $X$ the set of pairs $(x, c) \in [a] \times [2b+1]$ for which there exists $x'$ so that $(\{x, x'\}, c) \in C$. Note that each element $(\{x, y\}, c)$ gives rise to at most two such pairs in $X$, namely, $(x, c)$ and $(y, c)$ provided both $x, y \ne 0$. Hence, $|X| \leq 2|C|$. \\

\noindent
On the other hand, we claim that for each $x \in [a]$ there must be at least $b+1$ elements $c$ of $[2b+1]$ with $(x,c) \in X$. If this were not the case, then we could find a set $S$ with at least $(2b+1)-b = b+1$ elements so that no element in $X$ is of the form $(x, c)$ with $c \in S$. However, this would imply that $(x, S)$ is isolated, a contradiction. \\

\noindent
It follows that $2|C| \geq |X| \geq a \cdot (b+1)$, and thus to ``cover'' all vertices one must use at least $\frac{1}{2}a(b+1)$ colours, proving the lemma.
\end{proof}
\noindent
Now, let $n, r$ be positive integers so that $\sqrt{2r} < n < (1.01)^r$. Define $x = n^2/r$, so that $x > 2$ and we have to establish $B(n, r) = \Omega \left(\sqrt{r \log x}\right)$. We saw in the previous case that $B(\lfloor \sqrt{2r} \rfloor, r) = \Omega(\sqrt{r})$ and thus for $x \in (2,16]$, that is, for $\sqrt{2r} \leq n \leq 4 \sqrt{r}$ we have $B(n, r) = \Omega(\sqrt{r \log x})$ since $B$ is increasing in its first argument. Thus, in what follows, we may assume $x > 16$. \\

\noindent
We choose $b = \lfloor \log_{16} x \rfloor$ and $a = \left\lfloor \sqrt{\frac{r}{4(2b+1)}} \right\rfloor$. Note that $b \geq 1$ as $x > 16$, and since $n < (1.01)^r$, we have \[b = \lfloor \log_{16}(n^2/r) \rfloor \leq \log_{16}(n^2) \leq 2r\log_{16} (1.01) \leq \frac{r}{12}\]
    and thus $r \geq 8b + 4$, implying $a \geq 1$. We also have \[a \binom{2b+1}{b+1} \leq (r/4)^{1/2} \cdot 2^{2b+1} =r^{1/2} \cdot 4^b \leq r^{1/2} (n^2/r)^{\log_{16} 4} = n,\]
    and \[\binom{a+1}{2} (2b+1) \leq a^2 \cdot (2b+1) \leq \frac{r}{4(2b+1)} \cdot (2b+1) \leq r.\]
    From \hyperref[lem:3:1]{Lemma 3.1} and \hyperref[lem:3:2]{Lemma 3.2}, it follows that $B(n, r) \geq \frac{1}{2}a(b+1)$. Since $a = \Theta \left(\sqrt{r/b}\right)$ and $b = \Theta(\log x)$, it follows that $ab = \Theta \left(\sqrt{r \log x}\right)$ and hence $B(n, r) = \Omega\left(\sqrt{r \log x} \right) = \Omega \left(\sqrt{r \log(n^2/r)}\right)$, which settles this regime.
\subsection*{Case III: $n \geq (1.01)^r$}
From the previous Case, we saw that $B(\lfloor(1.01)^r\rfloor, r) = \Omega \left(\sqrt{r \log (((1.01)^r)^2/r)}\right) = \Omega(r)$. Since $B$ is increasing in the first coordinate, it follows that $B(n, r) = \Omega(r)$ for all $n \ge (1.01)^r$, which settles this regime as well.\\

\noindent
While the above three cases together prove the desired lower bounds in \hyperref[thm:1:1]{Theorem 1.1}, we provide an alternate randomised construction for the regime $r^{1/2+\varepsilon} \leq n < (1.01)^r$ for any $\varepsilon \in (0,0.1)$. Note that in this regime our goal becomes to show that $B(n, r) = \Omega_{\varepsilon} \left(\sqrt{r \log n}\right)$. We assume $r$ to be sufficiently large in what follows.\\

\noindent
\textbf{\underline{\textit{Alternate Construction for the Lower Bound}}} \\
\noindent
Stipulate $s = \left \lceil 2\sqrt{ r \log n} \right\rceil $. Sample sets $S_1, S_2, \dots, S_n$ independently and uniformly at random from $\binom{[r]}{s}$. We show that with probability more than $1-o_r(1)$ each, both the following events occur: 
\begin{itemize}
    \item For any two sets $S_i, S_j$ we have $S_i \cap S_j \ne \varnothing$.
    \item No set of size $\leq 0.1\varepsilon \sqrt{r \log n}$ intersects all sets $S_i$.
\end{itemize}
The probability that two randomly chosen subsets of size $s$ of $[r]$ are disjoint (for $s < r/2$) is given by \[\frac{\binom{r-s}{s}}{\binom{r}{s}} = \frac{(r-s)(r-s-1)\dots(r-2s+1)}{r(r-1)\dots (r-s+1)} = \prod \limits_{x = r-s+1}^{r} \left(1 - \frac{s}{x}\right) \leq \left(1-\frac{s}{r}\right)^s \leq \exp(-s^2/r),\]
and thus from a union bound the probability that the first condition fails is at most $n^2 \exp(-s^2/r) \leq n^{-2} = o_n(1) = o_r(1)$. \\

\noindent
Denote $t = \lfloor 0.1\varepsilon \sqrt{r \log n}\rfloor$. For any given subset of $[r]$ of size $t$, the probability that it intersects a given $S_i$ is given by (note that $2s+10t < r$ due to $n < (1.01)^r$ and $\varepsilon < 0.1$) \[1 - \frac{\binom{r-t}{s}}{\binom{r}{s}} = 1- \prod \limits_{x = r-s+1}^{r} \left(1 - \frac{t}{x}\right) \leq 1 - \left(1 - \frac{t}{r-s+1}\right)^s \leq 1-\left(1 - \frac{2t}{r}\right)^s \leq 1 - e^{-4st/r},\]
where at the end we use the estimate $1-x \geq e^{-2x}$ for $x \in (0,0.5)$. Since $st \leq 0.4\varepsilon r \log n$, the above probability is no more than $1-n^{-1.6\varepsilon}$. Thus the probability a given set of size $t$ meets all sets $S_1, S_2, \dots, S_n$ is at most $\left(1 - n^{-1.6\varepsilon}\right)^n$ and thus from a union bound the probability that the second condition fails is at most \[r^t (1-n^{-1.6 \varepsilon})^n \leq r^{0.1\varepsilon \sqrt{r \log n}} \exp\left(-n^{(1-1.6\varepsilon)}\right) \leq e^{0.1\varepsilon\sqrt{r \log n} \log r } \exp \left(-n^{(1-1.6\varepsilon)}\right).\]
Since $n > r^{1/2+\varepsilon}$, $\sqrt{r \log n} \log r \leq n^{\frac{1}{1+2\varepsilon} +o_{n,r}(1)}$ and thus the probability of the second condition failing is \[\leq e^{0.1\varepsilon n^{\frac{1}{1+2\varepsilon}+o_{n,r}(1)}} \exp\left(-n^{(1-1.6\varepsilon)}\right) = \exp\left(0.1\varepsilon n^{\frac{1}{2+\varepsilon} + o_{n,r}(1) } - n^{(1-1.6\varepsilon)}\right) = o_{n,r}(1) = o_r(1),\]
where we use $\varepsilon < 0.1$ to deduce $1-1.6\varepsilon > \frac{1}{1+2\varepsilon}$. (note that $n > r^{1/2+\varepsilon}$ implies $o_{n,r}(1) = o_{r}(1)$) \\

\noindent
Thus, each of the two events has probability of failure $o_r(1)$, and in particular, by another union bound, for large enough $r$ we can find sets $S_1, S_2, \dots, S_n \in \binom{[r]}{s}$ that obey the two properties. Define now the $r$-edge colouring $\chi : \binom{[n]}{2} \mapsto [r]$ as $\chi(\{i, j\}) = \min(S_i \cap S_j)$ for all $1 \leq i < j \leq n$ (as earlier, the choice of $\min$ can be made arbitrary). Since any two of the sets $S_i$ intersect, this colouring is well defined, and moreover one must select $> 0.1\varepsilon\sqrt{r \log n}$ colours to ensure that no vertex remains isolated, as no subset of $[r]$ with size $\leq 0.1\varepsilon \sqrt{r \log n}$ intersects all the $S_i$. Thus this colouring exhibits $B(n, r) = \Omega_{\varepsilon}(\sqrt{r \log n})$.
\section{Further Remarks and an Exact Threshold}
\label{sec:4}
Recall that $B(n,r)$ is nondecreasing in the first coordinate. Since one always has $B(n,r)\leq r$, the limit
\[
\lim_{n\to\infty} B(n,r)
\]
exists and is finite. In fact, this limit is exactly $\lceil r/2\rceil$, and we give a short argument for this below. In particular, this also explains why the family $\binom{[r]}{\lceil r/2\rceil}$ naturally appears later on. \\

\noindent
By arbitrarily dividing the set of colours into two parts of roughly equal size, and applying the classical result of Erd\H{o}s and Rado, we obtain the upper bound
$
A(n,r)\leq \lceil r/2 \rceil,
$
and thus also
$
B(n,r)\leq \lceil r/2 \rceil.
$ \\

\noindent
For the lower bound, when $r$ is odd we take the vertex set $V = \binom{[r]}{\lceil r/2 \rceil}$
and colour the edge $S_1S_2 \in \binom{V}{2}$ with any element of $S_1\cap S_2$. In this construction, one needs at least $\lceil r/2 \rceil$ colours to ensure that no vertex is isolated: indeed, any set of colours of size $\lfloor r/2 \rfloor$ leaves exactly one vertex isolated, namely the one described by its complement. Thus the limit of $B(\cdot,r)$ is $\lceil r/2 \rceil$ when $r$ is odd, and slight tweaks of this argument extend the statement to even $r$ as well. Since $A(n,r)\geq B(n,r)$, it follows that
\[
\lim_{n\to\infty} A(n,r)=\lceil r/2 \rceil = \lim \limits_{n \to \infty} B(n,r)
\]
as well. \\

\noindent
A more interesting question is to determine the least $n$ for which this eventual limit is attained. That is, given $r\geq 1$, what is the least $n$ such that $B(n,r)=\lceil r/2 \rceil$? It turns out that when $r$ is odd, the construction above is optimal, and the least such $n$ is the central binomial coefficient $
\binom{r}{\lceil r/2 \rceil}.$ The even-$r$ case appears to be similar, but we do not pursue it here.\\

\noindent
The argument to show the above is straightforward: note that a given vertex is isolated by at most one subset of size $\lfloor r/2 \rfloor$, but each of the $\binom{r}{\lfloor r/2 \rfloor}$ choices of colours is known to give some isolated vertex, and thus we obtain $n \geq \binom{r}{\lfloor r/2 \rfloor}$. One can ask the same question for $A$, and this is even more interesting. Clearly this answer is at most the corresponding answer for $B$, but it is not obvious whether the two always coincide.\\

\noindent
Interestingly, the two answers do coincide. More precisely, one has: 

\begin{theorem}
\label{thm:4:1}
For any odd integer $r \geq 3$, the least $n$ for which $A(n,r)= \lceil \frac{r}{2} \rceil$ is $\binom{r}{\lceil r/2 \rceil}$.
\end{theorem}
\noindent
We record here an adaptation of a proof by Abel George Mathew, communicated to us personally. 
\begin{proof}
An easy check settles the case $r = 3$. Henceforth, we write $r = 2k+1$ where $k \geq 2$. Suppose $n$ is a positive integer such that $A(n, 2k+1) = k+1$, and let $\chi : \binom{[n]}{2} \mapsto [2k+1]$ be a colouring that witnesses this and denote $M = \binom{2k+1}{k}$. Recall that for $S \subseteq [2k+1]$, we denote by $G_{\chi}[S]$ the subgraph of the $K_n$ formed by the edges whose colour lies in $S$. \\

\noindent
For each of the $M$ sets of colours in the family $\binom{[2k+1]}{k}$, the graph formed by those colours is disconnected. It follows that we can find partitions $[n] = A_i \sqcup B_i$ for $1 \leq i \leq M$ such that all edges in $E(A_i, B_i)$ draw their colours from $T_i^{c}$ where the $\{T_i\}_{1 \leq i \leq M}$ are the sets from $\binom{[2k+1]}{k}$ indexed in some order. Henceforth, we will let $\{S_i\}_{1 \leq i \leq M}$ denote the sets $S_i = T_i^{c}$ so that $S_1, \dots, S_M$ are all the subsets of $[2k+1]$ of size $k+1$.
\begin{claim}
\label{claim:4:2}
For any $1 \leq i \ne j \leq M$, the partitions $A_i \sqcup B_i$ and $A_j \sqcup B_j$ are ``parallel'', that is, $C_i \cap C_j = \varnothing$, for some choice $C_i \in \{A_i, B_i\}$ and $C_j \in \{A_j,B_j\}$.
\end{claim}
\begin{proof}
Suppose, for the sake of contradiction, that all four sets 
\begin{align*}
    & X_{00} = A_i \cap A_j \\ & X_{01} = A_i \cap B_j \\ & X_{10} = B_i \cap A_j \\ &X_{11} = B_i \cap B_j
\end{align*}
are nonempty. Note that all edges in $E(X_{00}, X_{11})$ receive colours from $S_i \cap S_j$, and the same is true for $E(X_{01}, X_{10})$. Thus $G_{\chi}[S_i \cap S_j]$ is a graph with at most two connected components, and so $A(n,r) \leq |S_i \cap S_j| + 1$, which implies $|S_i \cap S_j| \geq k$. \\

\noindent
Since $S_i \ne S_j$ and $|S_i| = |S_j| = k+1$, it follows that $S_i = S' \cup \{c_1\}$ and $S_j = S' \cup \{c_2\}$ for some set $S'$ of size $k$ and colours $c_1 \ne c_2$. Now, since $G_{\chi}[S']$ has exactly two connected components, no edge from $E(X_{00} \cup X_{11}, X_{01} \cup X_{10})$ has colour in $S'$. But any edge in this set certainly has colour in $S_i \cup S_j$, and thus the set of edges across this cut draws its colours from $\{c_1, c_2\}$. But now $G_{\chi}[\{c_1,c_2\}]$ is connected, a contradiction to $A(n,2k+1) = k+1 > 2$.
\end{proof}
\noindent
Call a set $A_i$ (or $B_i$) \emph{good} if there exists $j \ne i$ such that either $A_i \cap A_j$ or $A_i \cap B_j$ is empty. \hyperref[claim:4:2]{Claim 4.2} tells us that for all $i$, one of $A_i$ or $B_i$ is good.
\begin{claim}
\label{claim:4:3}
Fix $i \in [M]$. If $A_i$ (or $B_i$) is good, then there exists $v \in B_i$ (or $v \in A_i$) such that all edges $(u, v)$ for $u \in A_i$ are the same colour.
\end{claim}
\begin{proof}
Suppose that $A_i$ is good, and without losing generality suppose that $A_i \cap A_j = \varnothing$ for $j \ne i$. It follows that there is a partition $[n]$ as $A_i \sqcup X \sqcup A_j$ so that $B_i = X \cup A_j$ and $B_j = X \cup A_i$. Pick $c$ to be an arbitrary colour from $S_i \setminus S_j$. \\

\noindent
Note that $G_{\chi}[S_i \setminus \{c\}]$ contains all edges across $A_i$ and $A_j$ since all these edges have colours in $S_i \cap S_j$ and $c \not \in S_j$. In particular all vertices in $A_i \cup A_j$ are in the same connected component. Since $S_i \setminus \{c\}$ has size $k$ and $A(n,2k+1) = k+1$, we can find some $x \in X \subseteq B_i$ not in this connected component. It follows that no edge $(u, x)$ for $u \in A_i$ is present in $G_{\chi}[S_i \setminus \{c\}]$. But since $\chi(u,x) \in S_i$ for all such $u$, it follows that $\chi(u,x) = c$ for all $u \in A_i$, and so $v := x$ satisfies the required conditions.
\end{proof}
\begin{claim}
For all $i \in [M]$ \emph{exactly} one of $A_i$ and $B_i$ is good.
\end{claim}
\begin{proof}
We know already that at least one of the two is good. If both are good, then from \hyperref[claim:4:3]{Claim 4.3} we can find $u \in A_i$ and $v \in B_i$ such that $\chi$ is constant on $E(\{u\}, B_i)$ and $E(A_i, \{v\})$. But now these two colours are both equal to $\chi(u,v)$, and this single colour connects the graph, contradicting $A(n,2k+1) = k+1  > 1$.
\end{proof}
\noindent
Using the above, we can relabel $A_i, B_i$ so that all the $A_i$ are good. It follows that $A_i \cap A_j = \varnothing$ for all $i \ne j$. Since $A_i \subseteq V(G)$ for all $i$ and each $A_i$ is nonempty, we deduce $|V(G)| \geq M$, or that $n \geq \binom{2k+1}{k}$. We already argued that $A(M, 2k+1) = k+1$, so this finishes the proof.
\end{proof}
\noindent
\section{Acknowledgements and AI usage}
We are grateful to Noah Kravitz for pointing us to \cite{ABCK2024}, for his detailed reading of the paper, and for providing many helpful comments and suggestions that have improved this work. We also thank Abel George Mathew for providing a proof for \hyperref[thm:4:1]{Theorem 4.1}. We used ChatGPT as a sounding board for ideas, especially in the choices of explicit constants in the lemmas. It also provided proofreading assistance and help with the literature review. 
{\small
\bibliographystyle{alpha}
\bibliography{citations}
}
\begin{center}
{\scshape St John's College, University of Oxford; St Giles', Oxford OX1 3JP, UK\par}
\vspace{0.2em}
{\itshape Email address:} \texttt{archit.manas@sjc.ox.ac.uk}
\end{center}
\end{document}